\documentclass[10pt,reqno]{amsart}
\usepackage{amssymb}
\usepackage{amsxtra}
\usepackage[mathscr]{eucal}
\usepackage{srcltx}

\DeclareMathOperator*{\Cat}{\mathbf{Cat}}

\newcommand{\eop}{\hfill$\square$}

\theoremstyle{plain}
\newtheorem{Thm}{Theorem}[section]

\newtheorem{Cor}[Thm]{Corollary}
\newtheorem{Lem}[Thm]{Lemma}

\theoremstyle{definition}

\newtheorem{Rem}{Remark}[section]
\theoremstyle{remark}

\begin{document}

\title[Xia Zhou and David M.~Bradley]{On Mordell-Tornheim Sums and Multiple Zeta Values}

\date{\today}

\dedicatory{To Paulo Ribenboim, in honor of his 80th birthday.}

\author{David~M. Bradley}
\address{Department of Mathematics \& Statistics\\
         University of Maine\\
         5752 Neville Hall
         Orono, Maine 04469-5752\\
         U.S.A.}
\email[]{bradley@math.umaine.edu, dbradley@member.ams.org}

\author{Xia Zhou}
\address{Department of Mathematics\\ Zhejiang University\\ Hangzhou,
310027\\ P.\ R.\ China} \email[]{xiazhou0821@hotmail.com}
\thanks{Research of the second author was supported by the National Natural Science Foundation of China, Project
10871169.}



\begin{abstract}
We prove that any Mordell-Tornheim sum with positive integer
arguments can be expressed as a rational linear combination of
multiple zeta values of the same weight and depth.  By a result of
Tsumura, it follows that any Mordell-Tornheim sum with weight and
depth of opposite parity can be expressed as a rational linear
combination of products of multiple zeta values of lower depth.
\end{abstract}

\maketitle

\interdisplaylinepenalty=500

\section{Introduction}
Let $r$ and $w$ be positive integers, and let $s_1,\dots,s_r$ and
$s$ be complex numbers satisfying $s_1+\cdots+s_r+s=w$. A
Mordell-Tornheim sum of depth $r$ and weight $w$ is a multiple
series of the form
\begin{equation}\label{MTdef}
   T(s_1,\dots,s_r;s) := \sum_{m_1=1}^\infty \cdots \sum_{m_r=1}^\infty
   \frac{1}{m_1^{s_1}\cdots m_r^{s_r}(m_1+\cdots+m_r)^s}.
\end{equation}
Denote the real part of $s$ by $\sigma$, and the real part of $s_j$
by $\sigma_j$ for $1\le j\le r$.  Since~\eqref{MTdef} remains
unchanged if the arguments $s_1,\dots,s_r$ are permuted, we may as
well suppose that they are arranged in order of increasing real
part.  Then $\sigma_1\le \sigma_2\le\cdots\le \sigma_r$, and by
Theorem~\ref{thm:MTAC} below, the series~\eqref{MTdef} is absolutely
convergent if
\[
   \sigma+\sum_{j=1}^k \sigma_j > k
\]
for each $k=1,2,\dots,r$. We call~\eqref{MTdef} a Mordell-Tornheim
zeta value in the case when the arguments are all integers. These
were first investigated by Tornheim~\cite{Torn} in the case $r=2$,
and later by Mordell~\cite{Mord} and Hoffman~\cite{Hoff} with
$s_1=\cdots=s_r=1$.

%

Of greater theoretical importance are the so-called multiple zeta
series of depth $r$ and weight $w=s_1+\cdots+s_r$ of the form
\begin{equation}\label{MZVdef}
   \zeta(s_1,\dots,s_r) := \sum_{n_1>\cdots>n_r>0}\; \prod_{j=1}^r
   n_j^{-s_j},
\end{equation}
in which the sum is over all positive integers $n_1,\dots, n_r$ such
that $n_j>n_{j+1}$ for $1\le j\le r-1$.   By Theorem~\ref{thm:MZVAC}
below the series~\eqref{MZVdef} is absolutely convergent if the
partial sums of the real parts of the arguments satisfy
\[
   \sum_{j=1}^k \Re(s_j)>k
\]
for each $k=1,2,\dots,r$.  If $s_1,\dots,s_r$ are all integers,
then~\eqref{MZVdef} is called a multiple zeta value of depth $r$ and
weight $s_1+\cdots+s_r$.  Multiple zeta values have been studied
extensively;
see~\cite{DJBor,GoldBach,BBB,BBBLc,BBBLa,BowBradSurvey,BowBrad3,BowBrad1,BowBradRyoo,Prtn,DBqMzv,qEulerDecomp,qSum}
for example.

In this paper, we show how to express an \emph{arbitrary}
Mordell-Tornheim zeta value in terms of multiple zeta values of the
same weight and depth.  More precisely, we shall prove the following
result.

\begin{Thm}\label{thm:MTtoMZV}
  Every Mordell-Tornheim zeta value of depth $r$ and weight $w$ can be expressed as a rational linear
  combination of multiple zeta values of depth $r$ and weight $w$.
\end{Thm}

Theorem~\ref{thm:MTtoMZV} shows that the study of Mordell-Tornheim
zeta values reduces to the study of multiple zeta values. For
example, the following parity result is an immediate consequence of
Theorem~\ref{thm:MTtoMZV} and the corresponding parity result for
multiple zeta values due to Tsumura~\cite{Tsu02} and for which an
independent proof is given in~\cite{IKZ}.

\begin{Cor}\label{cor:MTtoMZVred}
  Every Mordell-Tornheim zeta value of depth at least 2 and with weight
  and depth of opposite parity can be expressed as a rational linear
  combination of products of multiple zeta values of lower depth.
\end{Cor}

We note that the case $r=2$ of Theorem~\ref{cor:MTtoMZVred} was
proved by Tornheim~\cite{Torn}.  Explicit formulas for Tornheim's
reduction were given in~\cite{HWZ}; see also~\cite{ZhouCaiBradley1}.

\section{Convergence Criteria}
\begin{Thm}\label{thm:MZVAC}
  Let $r$ be a positive integer, and let $s_1,\dots,s_r$ be complex numbers with respective
  real parts $\sigma_1,\dots,\sigma_r$.  The multiple zeta series~\eqref{MZVdef} is absolutely convergent if for each
  positive integer $k$ such that $1\le k\le r$, the inequality
  \[
     \sum_{j=1}^k \sigma_j>k
  \]
  holds.
\end{Thm}

\noindent{\bf Proof.}  The case $r=1$ is a familiar consequence of
the integral test from calculus.  Let $d$ be a positive integer, and
let $s_1,s_2,\dots,s_{d+1}$ be complex numbers with respective real
parts $\sigma_1,\sigma_2,\dots,\sigma_{d+1}$. First, suppose that
$\sigma_{d+1}<1$.  The Euler-Maclaurin sum formula implies that
\begin{equation}\label{MZVACcase1}
   \sum_{n_1>\cdots>n_{d+1}>0}\;\bigg|\prod_{j=1}^{d+1}
   n_j^{-s_j}\bigg|
   \ll \sum_{n_1>\cdots>n_d>0}\; n_d^{1-\sigma_{d+1}}\prod_{j=1}^d
   n_j^{-\sigma_j}.
\end{equation}
By induction, the series on the right hand side
of~\eqref{MZVACcase1} converges if for each positive integer $k$
such that $1\le k\le d-1$,
\[
   \sum_{j=1}^k \sigma_j > k  \quad\mbox{and}\quad
   \bigg[(\sigma_d+\sigma_{d+1}-1)+\sum_{j=1}^{d-1} \sigma_j  >d
   \quad\Longleftrightarrow\quad \sum_{j=1}^{d+1} \sigma_j > d+1\bigg].
\]
Therefore, the series obtained by removing the absolute value bars
from the series on the left hand side of~\eqref{MZVACcase1} is
absolutely convergent \emph{a fortiori} if for each positive integer
$k$ such that $1\le k\le d+1$,
\[
   \sum_{j=1}^k \sigma_j > k.
\]

Now suppose that $\sigma_{d+1}\ge 1$ and that
\[
   \sum_{j=1}^k \sigma_j > k
\]
for every positive integer $k$ such that $1\le k\le d+1$. Let
$\varepsilon>0$ be defined by the equation
\[
   \sum_{j=1}^d \sigma_j =d+2\varepsilon.
\]
The Euler-Maclaurin sum formula implies that
\begin{equation}\label{MZVACcase2}
   \sum_{n_1>\cdots>n_{d+1}>0}\;\bigg|\prod_{j=1}^{d+1} n_j^{-s_j}\bigg|
   \ll \sum_{n_1>\cdots>n_d>0}\; (\log n_d)\prod_{j=1}^d  n_j^{-\sigma_j}
   \ll \sum_{n_1>\cdots>n_d>0}\; n_d^{-(\sigma_d-\varepsilon)}\prod_{j=1}^{d-1}
   n_j^{-\sigma_j}.
\end{equation}
By induction, the series on the right hand side
of~\eqref{MZVACcase2} converges because for each positive integer
$k$ such that $1\le k\le d-1$,
\[
   \sum_{j=1}^k \sigma_j >k \quad\mbox{and}\quad
   (\sigma_d-\varepsilon)+\sum_{j=1}^{d-1}\sigma_j = d+\varepsilon
   >d.
\]
Therefore, the series obtained by removing the absolute value bars
from the series on the left hand side of~\eqref{MZVACcase2} is
absolutely convergent. \eop

\begin{Rem}  The condition for absolute convergence of~\eqref{MZVdef} is
incorrectly stated in~\cite{Zhao} as
\[
   \sigma_1>1\quad\mbox{and}\quad\sum_{j=1}^r \sigma_j > r.
\]
For a counterexample, these inequalities are satisfied if
$s_1=s_3=2$ and $s_2=0$, but
\[
   \sum_{n_1>n_2>n_3>0}\; n_1^{-2} n_3^{-2} = \sum_{n_1=1}^\infty
   \frac{1}{n_1^2}\sum_{n_2=1}^{n_1-1}\,\sum_{n_3=1}^{n_2-1}
   \frac{1}{n_3^2} \ge \sum_{n=3}^\infty
   \frac{1}{n^2}\sum_{k=2}^{n-1} 1 = \sum_{n=3}^\infty
   \frac{n-2}{n^2} = \infty.
\]
Sufficient conditions for absolute convergence of a more general
class of multiple Dirichlet series is given in~\cite{Kohji}, but the
proof takes 4 pages.
\end{Rem}

\begin{Thm}\label{thm:MTAC}
  Let $r$ be a positive integer, and let $s_1,\dots,s_r$ be complex
  numbers arranged so that their respective real parts
  $\sigma_1,\dots,\sigma_r$ satisfy $\sigma_j\le \sigma_{j+1}$ for
  $1\le j\le r-1$.  Let $s$ be a complex number with real part
  $\sigma$.  If for each positive integer $k$ such that $1\le k\le
  r$, the inequality
  \[
     \sigma+\sum_{j=1}^k \sigma_j >k
  \]
  holds, then the Mordell-Tornheim series~\eqref{MTdef} is
  absolutely convergent.
\end{Thm}

\noindent{\bf Proof.}  The summation indices $m_1,\dots,m_r$
in~\eqref{MTdef} obviously satisfy
\[
   \max \{m_j: 1\le j\le r\} \le \sum_{j=1}^r m_j \le r\max\{ m_j :
   1\le j\le r\}.
\]
Therefore, using the symbol $\asymp$ as a short-hand for ``has the
same order of magnitude as'', we have
\begin{equation}\label{permute}
   \sum_{m_1=1}^\infty \cdots\sum_{m_r=1}^\infty \bigg|
   \frac{1}{m_1^{s_1}\cdots m_r^{s_r}(m_1+\cdots+m_r)^s}\bigg|
   \asymp \sum_{\pi\in\mathfrak{S}_r}
   \sum_{m_{\pi(1)}>\cdots>m_{\pi(r)}>0}\;
   m_{\pi(1)}^{-\sigma}\prod_{j=1}^r m_{\pi(j)}^{-\sigma_{\pi(j)}},
\end{equation}
where the outer sum on the right is over all permutations $\pi$ of
$\{1,2,\dots,r\}$ and we have ignored all cases where there exists
an equality between two or more indices $m_j$ since these series
converge under less stringent conditions.   By
Theorem~\ref{thm:MZVAC}, the series on the right hand side
of~\eqref{permute} is absolutely convergent if for each permutation
$\pi$ and each positive integer $k$ such that $1\le k\le r$, we have
\[
   \sigma+\sum_{j=1}^k \sigma_{\pi(j)} >k.
\]
Since $\sigma_1\le \sigma_2\le\cdots\le \sigma_r$, this will clearly
be the case if for each positive integer $k$ such that $1\le k\le
r$, we have
\[
   \sigma+\sum_{j=1}^k \sigma_j>k.
\]
\eop

\section{Proof of Theorem~\ref{thm:MTtoMZV}}
Key to our proof of Theorem~\ref{thm:MTtoMZV} is the following
partial fraction decomposition.
\begin{Lem}\label{lem:parfrac}
Let $r$ and $s_1,s_2,\dots,s_r$ be positive integers, and let
$x_1,x_2,\dots,x_r$ be non-zero real numbers such that
$x:=x_1+x_2+\cdots+x_r\ne 0$. Then
\[
   \prod_{j=1}^r x_j^{-s_j}
   = \sum_{j=1}^r \bigg(\prod_{\substack{k=1\\k\ne j}}^r
  \sum_{a_k=0}^{s_k-1}\bigg)M_j\, x^{-s_j-A_j}\prod_{\substack{k=1\\k\ne j}}^r
  x_k^{a_k-s_k},
\]
where the multinomial coefficient
\[
   M_j := \frac{(s_j+A_j-1)!}{(s_j-1)!}\prod_{\substack{k=1\\k\ne
   j}}^r \frac{1}{a_k!}
   \qquad\mbox{and}\qquad
   A_j := \sum_{\substack{k=1\\k\ne j}}^r a_k.
\]
\end{Lem}

\begin{proof} Applying the partial differential operator
\[
   \prod_{n=1}^r \frac{1}{(s_n-1)!}\bigg(-\frac{\partial}{\partial
   x_n}\bigg)^{s_n-1}
\]
to both sides of the trivial identity
\[
   \prod_{j=1}^r x_j^{-1} = \sum_{j=1}^r
   x^{-1}\prod_{\substack{k=1\\k\ne j}}^r
   x_k^{-1},\quad\mbox{with}\quad
   x:=\sum_{j=1}^r x_j
\]
yields
\begin{align*}
  \prod_{j=1}^r x_j^{-s_j} &= \sum_{j=1}^r
  \bigg\{\prod_{\substack{n=1\\n\ne j}}^r
  \frac{1}{(s_n-1)!}\bigg(-\frac{\partial}{\partial
  x_n}\bigg)^{s_n-1}\bigg\} x^{-s_j}\prod_{\substack{k=1\\k\ne j}}^r
  x_k^{-1}\\
  &= \sum_{j=1}^r \bigg(\prod_{\substack{k=1\\k\ne j}}^r
  \sum_{a_k=0}^{s_k-1}\bigg)\bigg(\frac{(s_j+A_j-1)!}{(s_j-1)!}\prod_{\substack{k=1\\k\ne
   j}}^r \frac{1}{a_k!} \bigg) x^{-s_j-A_j}\prod_{\substack{k=1\\k\ne j}}^r
  x_k^{a_k-s_k},
\end{align*}
as claimed.
\end{proof}

\noindent{\bf Proof of Theorem~\ref{thm:MTtoMZV}.} For $1\le l\le
r-1$, let
\[
   T_l(s_1,\dots,s_r) := \sum_{m_1=1}^\infty\dots\sum_{m_r=1}^\infty
   \bigg(\prod_{k=1}^l m_k^{-s_k}\bigg)\bigg(\prod_{k=l+1}^r
   n_k^{-s_k}\bigg),
   \quad\mbox{with}\quad n_k := \sum_{j=1}^k m_j.
\]
In Lemma~\ref{lem:parfrac}, let $x_j=m_j$, multiply both sides by
$n_r^{-s}$ and sum over all positive integers $m_j$ for $1\le j\le
r$. We find that
\begin{align}\label{step1}
   T(s_1,\dots,s_r;s) &= \sum_{j=1}^r
   \bigg(\prod_{\substack{k=1\\k\ne j}}^r \sum_{a_k=0}^{s_k-1}\bigg)
   M_j \sum_{m_1=1}^\infty\dots\sum_{m_r=1}^\infty
   n_r^{-s-s_j-A_j}\prod_{\substack{k=1\\k\ne j}}^r m_k^{a_k-s_k}\nonumber\\
   &=\sum_{j=1}^r
   \bigg(\prod_{\substack{k=1\\k\ne j}}^r \sum_{a_k=0}^{s_k-1}\bigg)
   M_j
   T_{r-1}(\Cat_{\substack{k=1\\k\ne j}}^r \{s_k-a_k\},s+s_j+A_j),
\end{align}
where
\[
   \Cat_{\substack{k=1\\k\ne j}}^r \{t_k\}
\]
abbreviates the concatenated argument sequence
$t_1,t_2,\dots,t_{j-1},t_{j+1},\dots,t_{r-1},t_r$.  Note that the
weight of $T$ in~\eqref{step1} is equal to the sum of the arguments
in $T_{r-1}$ on the right hand side.   Now apply
Lemma~\ref{lem:parfrac} with $r=l$, $x_j=m_j$, multiply both sides
by
\[
   \prod_{k=l+1}^r n_k^{-s_k}
\]
and sum over all positive integers $m_j$ for $1\le j\le r$. We find
that
\begin{align}\label{step2}
   T_l(s_1,\dots,s_r) &= \sum_{j=1}^l\bigg(\prod_{\substack{k=1\\k\ne
   j}}^l \sum_{a_k=0}^{s_k-1}\bigg)M_j \sum_{m_1=1}^\infty \cdots\sum_{m_r=1}^\infty\bigg(\prod_{\substack{k=1\\k\ne j}}^l
   m_k^{a_k-s_k}\bigg) n_l^{-s_j-A_j}\prod_{k=l+1}^r
   n_k^{-s_k}\nonumber\\
   &=\sum_{j=1}^l\bigg(\prod_{\substack{k=1\\k\ne j}}^l \sum_{a_k=0}^{s_k-1}\bigg)M_j
   T_{l-1}(\Cat_{\substack{k=1\\k\ne j}}^l \{s_k-a_k\},
   s_j+A_j,\Cat_{k=l+1}^r s_k).
\end{align}
Since
\[
   \sum_{\substack{k=1\\k\ne j}}^l (s_k-a_k)+s_j+A_j+\sum_{k=l+1}^r
   s_k = \sum_{k=1}^r s_k,
\]
the weight is preserved in~\eqref{step2}.  Since
\[
   T_1(s_1,\dots,s_r) = \sum_{m_1=1}^\infty\dots\sum_{m_r=1}^\infty m_1^{-s_1}\prod_{k=2}^r n_k^{-s_k}
   = \sum_{n_r>\cdots>n_1>0}^\infty\; \prod_{k=1}^r
   n_k^{-s_k} = \zeta(s_r,\dots,s_1),
\]
by induction the proof is complete. \eop

Before concluding, we note the following easy consequences of the
results proved in this section.

\begin{Cor}\label{cor:MultiEulerDecomp}
  Let $r-1$ and $s_j-1$ be positive integers for $1\le j\le r$. Let
  $M_j$ and $A_j$ be as in Lemma~\ref{lem:parfrac} and let $T_{r-1}$
  be as in the proof of Theorem~\ref{thm:MTtoMZV}.
  Then
  \[
     \prod_{j=1}^r \zeta(s_j) = \sum_{j=1}^r
     \bigg(\prod_{\substack{k=1\\k\ne j}}^r
     \sum_{a_k=0}^{s_k-1}\bigg) M_j
     T_{r-1}(\Cat_{\substack{k=1\\k\ne j}}^r \{s_k-a_k\}, s_j+A_j).
  \]
\end{Cor}

\noindent{\bf Proof.} Sum both sides of Lemma~\ref{lem:parfrac} over
all positive integers $x_1,\dots,x_r$.  \eop

Note that when $r=2$, Corollary~\ref{cor:MultiEulerDecomp} reduces
to Euler's decomposition~\cite{qEulerDecomp}
\[
   \zeta(s)\zeta(t) =
   \sum_{a=0}^{s-1}\binom{a+t-1}{t-1}\zeta(t+a,s-a)+\sum_{a=0}^{t-1}\binom{a+s-1}{s-1}\zeta(s+a,t-a).
\]

\begin{Cor}[Corollary 4.2 in~\cite{Mord}] For any positive integers $r$ and $s$,
\begin{equation*}
  T(\underbrace{1, 1,\dots, 1}_r ; s)=r!\, \zeta(s+1,\underbrace{1,\dots, 1}_{r-1}).
\end{equation*}
\end{Cor}

Now combining equations (30) and (31) in~\cite{BBB}, we have
\begin{align*}
 T(\underbrace{1, 1,\dots, 1}_r ; 1) &= r!\,\zeta(2,\underbrace{1,\dots, 1}_{r-1})=r!\,\zeta(r+1)\\
 \intertext{and}
 T(\underbrace{1, 1,\dots, 1}_r ; 2) &= r!\,\zeta(3,\underbrace{1,\dots, 1}_{r-1})
  =r!\,\bigg\{\frac{r+1}{2}\zeta(r+2)-\frac{1}{2}\sum_{k=1}^{r-1}\zeta(k+1)\zeta(r+1-k)\bigg\}.
\end{align*}

\section{Parity Results}

In the introductory section, we alluded to the following parity
result for multiple zeta values due to Tsumura~\cite{Tsu02} and for
which an an independent proof is given in~\cite{IKZ}.
\begin{Thm}[Tsumura~\cite{Tsu02}, also Ihara et al~\cite{IKZ}]\label{thm:MZVtoMZVred}
   Every multiple zeta value of depth at least two and with weight and depth of opposite
   parity can be expressed as a rational linear combination of
   products of multiple zeta values of lower depth.
\end{Thm}
Clearly, Corollary~\ref{cor:MTtoMZVred} is an immediate consequence
of Theorem~\ref{thm:MZVtoMZVred} and our Theorem~\ref{thm:MTtoMZV}.
Alternatively, we can prove Corollary~\ref{cor:MTtoMZVred} by
employing instead a recent parity result of Tsumura~\cite{TsuProc}
for Mordell-Tornheim zeta values:

\begin{Thm}[Tsumura~\cite{TsuProc}]\label{thm:MTtoMTred}
   Every Mordell-Tornheim zeta value of depth at least two and with weight and depth of opposite
   parity can be expressed as a rational linear combination of
   products of Mordell-Tornheim zeta values of lower depth.
\end{Thm}

Corollary~\ref{cor:MTtoMZVred} is clearly also an immediate
consequence of Tsumura's Theorem~\ref{thm:MTtoMTred} and our
Theorem~\ref{thm:MTtoMZV}.

\section*{Acknowledgments}
Thanks are due to the referee for valuable suggestions concerning
emphasis and organization.  Research of the second author was
supported by the National Natural Science Foundation of China,
Project 10871169.

\end{document}